\documentclass[11pt]{article}
\usepackage{amsmath,amsthm,amsfonts,amssymb,mathrsfs,epsfig,bm}
\usepackage[usenames]{color}

\usepackage[colorlinks=true]{hyperref}
\usepackage{graphicx}

\usepackage[top=3cm, bottom=3cm, left=2cm, right=2cm, heightrounded,
 marginparwidth=2.5cm, marginparsep=3mm]{geometry}
\parskip        \medskipamount
\newtheorem{theorem}{Theorem}
\newtheorem{lemma}{Lemma}
\newtheorem{corollary}{Corollary}

\newtheorem{remark}{Remark}
\theoremstyle{remark}

\def\R{\mathbb{R}}

\def\P{\mathbb{P}}
\def\E{\mathbb{E}}

\renewcommand{\phi}{\varphi}
\renewcommand{\epsilon}{\varepsilon}

\newcommand{\cadlag}{{c\`adl\`ag} }

\definecolor{mygray}{gray}{0.9}
\definecolor{careful}{RGB}{250,150,150}
\definecolor{deeppink}{RGB}{255,20,147}
\definecolor{mygreen}{rgb}{0.05, 0.576, 0.03}
\definecolor{myred}{rgb}{0.768, 0.09, 0.09}

\long\def\symbolfootnote[#1]#2{\begingroup
\def\thefootnote{\fnsymbol{footnote}}\footnote[#1]{#2}\endgroup}

\newcommand{\ams}[2]{  \noindent {\footnotesize
             {\small \em AMS {\rm 2000} subject classifications.
           {\rm Primary {\sc #1}; secondary {\sc #2}} } } }

\def\CC{\mathscr{C}}
\usepackage{MnSymbol}

\usepackage{authblk}
\def\Beta{\operatorname{B}}
\usepackage{authblk}
\usepackage{blindtext}

\begin{document}

\title{\bf Does the ratio of Laplace transforms of powers of a function identify
the function?}
\author[1]{Takis Konstantopoulos \thanks{\noindent Takis Konstantopoulos would like to thank ICMS Edinburgh and the University of Edinburgh for their hospitality. }}

\author[1,2]{Linglong Yuan\thanks{\noindent Linglong Yuan acknowledges the support of the National Natural Science Foundation of China (Youth Program, Grant: 11801458), and the XJTLU RDF-17-01-39.}}

\affil[1]{University of Liverpool}
\affil[2]{Xi'an Jiaotong-Liverpool University}

\date{\normalsize 10 September 2020}
\maketitle


\abstract{We study the following question: 
if $f$ is a nonzero measurable function on $[0,\infty)$ and $m$ and $n$ distinct
nonnegative integers, does the ratio $\widehat{f^n}/\widehat{f^m}$ of the
Laplace transforms of the powers $f^n$ and $f^m$ of $f$ uniquely determine $f$? The answer is yes if one of $m, n$ is zero, by the inverse Laplace transform. Under some assumptions on the smoothness of $f$ we show that the answer in the general case is also
affirmative. 
The question arose from a problem in economics, specifically in auction theory
where $f$ is the cumulative distribution function of a certain random variable.
This is also discussed in the paper. 
}
\symbolfootnote[0]{\ams{44A10,26Axx}{91B70,91B26}}

\section{Introduction}
Let $f:[0, \infty)\to \R$ be a nonzero measurable function of exponential order, that is,
there are positive numbers $C$ and $c$ such that
$|f(x)| \le C e^{cx}$ for all $x$. Denote by
\[
\widehat f(\lambda) := \int_0^\infty e^{-\lambda x} f(x) dx
\]
the Laplace transform of $f$, 
defined and analytic for all  $\lambda >c$.
If $n$ is a positive integer then the $n$-th power $f^n$ of $f$ is also
of exponential order  and $\widehat{f^n}$ denotes its Laplace transform.
Let $m,n$ be nonnegative integers. 
Define
\[
H_{n,m}(f,\lambda) := \frac{\widehat{f^n}(\lambda)}{\widehat{f^m}(\lambda)}.
\]
The question of interest in this paper is the following:
\begin{quote}
{\bf Uniqueness question:} For given distinct nonnegative integers $n$ and $m$,
does knowledge of the function $H_{n,m}(f,\cdot)$ uniquely specify $f$?
\end{quote}
For $m, n >0$, both $\widehat{f^n}(\lambda)$ and $\widehat{f^m}(\lambda)$ are analytic when $\lambda$ ranges on the complex plane and the real part of $\lambda$ is large enough, see, e.g., \cite[Theorem 6.1]{D12}. Note that we do not consider the question of determining the explicit expression of $f$, which is much harder. Instead, we 
only study the uniqueness of $f$ given the function $H_{n,m}(f, \cdot)$.

Clearly, if $m=0$ (and similarly for $n=0$) then, by the inverse Laplace 
transform \cite{WIDDER}, we know $f^n$ and so we know 
$f$ if $n$ is odd.  
But if $n$ and $m$ are distinct positive integers, the problem seems to be hard.
We aim at giving an answer when we restrict $f$ to a certain class of functions.

It is easy to see that uniqueness, in strict sense, is impossible because
translations do not affect $H_{n,m}(f,\cdot)$. 
Suppose that, for some $a>0$, the function $f$ 
is identically $0$ on an interval $[0,a)$
and let 
\[
\theta_{-a} f(x) := f(x+a).
\]
Then 
\[
\widehat{\theta_{-a} f}(\lambda) = 
e^{\lambda a} \widehat f(\lambda).
\]
Clearly then,
\[
H_{n,m}(f,\cdot) = H_{n,m}(\theta_{-a} f, \cdot).
\]
So $H_{n,m}(f,\cdot)$ specifies $f$ up to a translation.
Hence, to obtain uniqueness, it is necessary to assume 
\begin{equation}
\label{fg0}
\inf\{x:f(x)\neq 0\}=0.
\end{equation}
Even under this condition, we cannot answer the problem in general,
i.e.\ under the sole assumption that the Laplace transform of $f$ exists.

The problem arose as a question in stochastic modeling in economics,
in particular in auction theory \cite{LPV00}. In this case, the function $f$ is the
cumulative distribution function of a certain random variable (see
Section \ref{auctions}) and hence it is a nondecreasing function.
However, the question of uniquely determining $f$ from $H_{n,m}(f,\cdot)$ 
appears to be much more general and hence of interest  independent
of the auction theory model.

Our first result concerns the case where  $f$ is a polynomial.

\begin{theorem}\label{poly}
Let $m,n$ be distinct positive integers and $f,g$ polynomials such that
\[
H_{n,m}(f,\cdot) = H_{n,m}(g,\cdot).
\]
If $n-m$ is odd, then $f$ is identical to $g$. If 
$n-m$ is even, then either $f$ is identical to $g$ or $f$ is identical
to $-g$.
\end{theorem}

For the general case, we shall restrict ourselves to functions
$f$ on $[0,\infty)$
that are right continuous and with left limits at each point
(the so called \cadlag functions in probability theory) and impose smoothness on the right.

We say that $f$ is {\em right analytic} at a point $a$ if the
right derivatives $f^{(i)}(a+)$ exist for all $i \ge 0$ and if there exists
$h>0$ such that, for all $a \le x < a+h$,
\[
f(x)=\sum_{i=0}^{\infty}f^{(i)}(a+)\frac{(x-a)^i}{i!}.
\]
The series on the right also converges
on $a-h<x<a+h$ (see, e.g., \cite[Prop.\ 1.1.1]{REAL}).
By \cite[Cor.\ 1.2.3]{REAL}, 
the function $g(x) :=\sum_{i=0}^{\infty}f^{(i)}(a+)\frac{(x-a)^i}{i!}$
is real analytic on $(a-h,a+h)$.
So $f$ is right analytic at $a$ if and only if there exists a function $g$ 
(depending on $a$)
which is real analytic on $(a-h,a+h)$ for some $h>0$ such that $f(x)=g(x)$ for all $x\in[a,a+h)$. 
The right analyticity only imposes smoothness on the right of a point. A \cadlag and right analytic function $f$ on $[0,\infty)$ can have countably many discontinuous 
points on a compact interval. For example, take 
$$f(x)=\frac{1}{2^n},\text{ if }x\in\left[1-\frac{1}{2^n},1-\frac{1}{2^{n+1}}\right), n=0,1,2,\ldots; \quad f(x)=2, \text{ if }x\geq 1.$$
The $f$ defined above is \cadlag and right analytic on $[0,\infty)$ with
discontinuities at  points  $1-\frac{1}{2^n}$, $n=1,2,\ldots$, and at $1$. 

We state and prove the uniqueness theorem making the additional
assumption that $f$ is nondecreasing. 
We conjecture the theorem remains true even without this assumption.
A strong indication for this is that it is not difficult to extend Theorem 
\ref{poly} from polynomials to entire functions--see Remark \ref{entire}.
However, for the purposes of the original problem, posed in relation to
auction theory, the monotonicity assumption is natural; see Section \ref{auctions}.

\begin{theorem}
\label{Mainthm}
Let $m,n$ be distinct positive integers. 
Suppose that  $f, g$ are nonnegative nondecreasing 
\cadlag functions, both right analytic at every point of $[0,\infty)$ and
of exponential order, and such that $f(x), g(x)>0$ for all $x>0$.
If
\[
H_{n,m}(f,\cdot) = H_{n,m}(g,\cdot),
\]
then $f=g$.
\end{theorem}

The paper is organized as follows.
Theorems \ref{poly} and \ref{Mainthm} are proved in Section \ref{uni}.
The application to auction theory, whence the problem
originally appeared, is presented in Section \ref{auctions}.
Open questions are summarized in the brief last section.

\section{The uniqueness question}
\label{uni}

We start with a preliminary observation.
For a nonnegative integer $k$,
denote by $\CC^k_0$  the class of functions $f$ that are $k$ times
differentiable at the origin; so $f \in \CC^0_0$ iff $f$ is continuous at $0$.
Let
\[
I(f) := \inf\{k \ge 0:\, f \in \CC^k_0,\, f^{(k)}(0) \neq 0\}.
\]
The observation is that if $f$ and $g$ have finite $I(f)$ and $I(g)$ then
$H_{n,m}(f,\cdot) = H_{n,m}(g,\cdot)$ implies that $I(f)=I(g)$.
We explain this in the following lemma. Before that, we recall the 
definition of the beta function:
\begin{equation}\label{Beta}
\Beta(\alpha,\beta) = \int_0^1 t^{\alpha-1} (1-t)^{\beta-1} dt
= \frac{\Gamma(\alpha)\Gamma(\beta)}{\Gamma(\alpha+\beta)},\quad  \alpha>0, \beta>0,
\end{equation}
noting that $\Gamma(n)=(n-1)!$ when $n$ is a positive integer.

\begin{lemma}
\label{lemaux}
Suppose that $f$ and $g$ are of exponential order
with $I(f)<\infty, I(g)<\infty$.
Let $m, n$ be distinct positive integers. Assume
$H_{n,m}(f,\cdot) = H_{n,m}(g,\cdot)$.
Then $I(f)=I(g)$. 
Let $k=I(f)=I(g)$. If $n-m$ is odd then $f^{(k)}(0) = g^{(k)}(0)$.
If $n-m$ is even then $|f^{(k)}(0)| = |g^{(k)}(0)|$.
\end{lemma}
\begin{proof}
The assumption $H_{n,m}(f,\cdot) = 
H_{n,m}(g,\cdot)$ is equivalent to
\[
\widehat{f^n}(\lambda) \, \widehat{g^m}(\lambda)
=\widehat{g^n\,}(\lambda) \, \widehat{f^m}(\lambda),\,  \text{ for sufficiently large } \lambda, 
\]
which is further equivalent to 
\begin{equation}
\label{conveq}
{f^n} * {g^m} = {f^m} * {g^n},
\end{equation}
where $*$ denotes convolution of functions on $[0,\infty)$, namely,
$(f*g)(t) = \int_0^t f(t-s) g(s) ds$.
Write the left-hand side as
\begin{equation}
\label{LHSconv}
({f^n} * {g^m}) (t)
= \int_0^t f(s)^n g(t-s)^m ds
= t \int_0^1 f(tu)^n g(t(1-u))^m du.
\end{equation}
Define
\[
 k:=I(f), \quad \ell := I(g),\quad
a:= f^{(k)}(0), \quad b:=f^{(\ell)}(0).
\]
Divide both sides of \eqref{LHSconv} by 
$t^{kn+\ell m+1}$. 
Then, as  $t \to 0$,
\begin{align}
\frac{(f^n * g^m) (t)}{t^{kn+\ell m+1}}
&= \int_0^1 \left(\frac{f(tu)}{t^k}\right)^n 
\left(\frac{g(t(1-u))}{t^\ell}\right)^m du
\nonumber
\\
&\to
\int_0^1  \left(\frac{a u^k}{k!}\right)^n 
\left(\frac{b (1-u)^\ell}{\ell!}\right)^m du
= \frac{a^n b^m}{k!^n \ell!^m} \Beta(kn+1, \ell m+1),  \label{D1}
\end{align}
where $\Beta$ is the beta function; see \eqref{Beta}.
To obtain this, we used the assumption
that the first nonzero derivative of $f$ at zero is the derivative of order $k$,
so that $f(tu)/t^k \to f^{(k)}(0) u^k/k!$ and,
similarly, $g(t(1-u))/t^\ell \to g^{(\ell)}(0) (1-u)^\ell/\ell!$.
Reversing the roles of $n$ and $m$, we obtain
\begin{equation}
\label{D2}
\frac{(f^m * g^n) (t)}{t^{km+\ell n+1}}
\to \frac{a^m b^n}{k!^m \ell!^n} \Beta(km+1, \ell n+1),
\end{equation}
as $t \to 0$.
Comparing \eqref{D1} and \eqref{D2}, and in view of \eqref{conveq}, 
we are forced to conclude that 
\[
p_1:=kn + \ell m ~=~ km +\ell n=:p_2.
\]
Indeed,  by \eqref{conveq}, we have $f^n*g^m = f^m*g^n=h$. The function $h$ satisfies 
$t^{-p_1} h(t) \to C_1$ and $t^{-p_2} h(t) \to C_2$, as $t\to 0$,
where $C_1, C_2$ are the
constants appearing on the right-hand sides of \eqref{D1} and \eqref{D2},
respectively. These constants are nonzero. If $p_1>p_2$ we obtain
$t^{-p_1} h(t)  = t^{p_1-p_2} (t^{-p_1} h(t)) \to 0$.
Hence $C_1=0$,
which is impossible. Similarly, $p_1<p_2$ is impossible, and thus $p_1=p_2$,
whence  $k(n-m) = \ell (n-m)$, and so
\[
k=\ell.
\]
But then $C_1$ and $C_2$ are equal and this entails
$a^m b^n = a^n b^m$, or
\[
(a/b)^{n-m}=1.
\]
If $n-m$ is odd we have $a=b$. If $n-m$ is even we can only deduce
that $|a|=|b|$.
\end{proof}

The next lemma shows equality of higher derivatives.

\begin{lemma}
\label{lem2}
Suppose that $f$ and $g$ are of exponential order
with $I(f)=I(g)=k<\infty$
and $f^{(k)}(0)=g^{(k)}(0)$.
Let $m, n$ be distinct positive integers. If
$H_{n,m}(f,\cdot) = H_{n,m}(g,\cdot)$
then $f^{(\ell)}(0)=g^{(\ell)}(0)$ for all $\ell \ge k$ for which
the two derivatives exist.
\end{lemma}

\begin{proof}
We prove that, for all $\ell \ge k$, $f, g \in \CC^\ell_0 \Rightarrow f^{(\ell)}(0)
=g^{(\ell)}(0)$, by induction on $\ell$.
Fix an integer $\ell > k$. The induction hypothesis is that
\[
\text{for all integers $j \in [k, \ell -1]$}, ~ 
f, g \in \CC^j_0 \Rightarrow f^{(j)}(0)=g^{(j)}(0).
\]
Assume that $f, g \in \CC^\ell_0$. 
To complete the inductive step, it remains to show 
that $f^{(\ell)}(0)=g^{(\ell)}(0)$.
For $k \le j \le \ell-1$, use the abbreviations
\[
c_j := f^{(j)}(0)/j!, \quad
a :=  f^{(\ell)}(0)/\ell!,\quad b :=  g^{(\ell)}(0)/\ell!,
\]
to write
\[
f(x)=\sum_{i=k}^{\ell-1} c_i x^i + a x^\ell + f_1(x),
\quad 
g(x)=\sum_{i=k}^{\ell-1} c_i x^i + b x^\ell + g_1(x),
\]
where $f_1(x) = o(x^\ell)$ and $g_1(x) = o(x^\ell)$ as $x \to 0$.
We will show that $a=b$.
We have
\begin{align}
\label{nmdev}
\frac{f^n*g^m(t)}{t}&=\int_0^1 f(tu) ^n g(t(1-u))^m du &\nonumber
\\&=
\int_0^1
\left( \sum_{i=k}^{\ell-1} c_i u^i t^i +  
a u^\ell t^\ell + f_1(ut) \right)^n &\nonumber
\\
&\quad\quad\quad\quad\quad\quad\quad\times \left(\sum_{i=k}^{\ell-1}c_i (1-u)^i t^i +
b(1-u)^\ell t^\ell+g_1((1-u)t) \right)^m du.&
\end{align}
Note that the
integrand in the last integral of \eqref{nmdev} is a product of $n+m$ terms. 
Let\footnote{Ignoring for the moment the terms $f_1$ and $g_1$,
so that the integrand is a polynomial, we can easily see that the term $t^d$
of this polynomial has a coefficient that depends on $a$ or $b$,
whereas all smaller degree terms do not.} 
\[
d=k(n+m-1)+\ell.
\]
After multiplication and integration, we shall keep track of the 
monomial terms of degree at most $d$ and 
combine everything else into terms of
order $o(t^d)$. Notice that if $f_1$ or $g_1$ is 
involved in the multiplication and integration, 
the resulting term must be of order $o(t^d).$ 
This  means that if we keep track of the 
monomial terms of degree at most $d$ then
$f_1$ and $g_1$ are not involved. 
These observations allow us to  write
\[
\frac{f^n*g^m(t)}{t} = P_{n,m}(t) + o(t^d).
\]
Note that $P_{n,m}(t)$ can be obtained as follows:  set $f_1$ and $g_1$ to zero
in the last integral of \eqref{nmdev}, integrate so that we
obtain a polynomial in $t$ of degree $n\ell+m\ell$,
and keep only the monomials up to power $t^d$.
We now split $P_{n,m}(t)$ into a polynomial $Q_{n,m}(t)$ of degree at most $d-1$
and a monomial  of degree $d$ whose coefficient is split into two parts:
\[
P_{n,m}(t) = Q_{n,m}(t)+ (C_{n,m}(a,b)+D_{n,m}) t^d.
\]
The first coefficient $C_{n,m}(a,b)$ contains all terms that depend
on $a$ or $b$. Explicitly,
\begin{align}
C_{n,m}(a, b) t^d 
&=\int_0^1 a u^\ell t^\ell
{n\choose 1} (c_k u^k t^k )^{n-1} (c_k (1-u)^k t^k)^m du&\nonumber\\ 
&\quad \quad \quad \quad \quad \quad +
\int_0^1 b (1-u)^\ell t^\ell
{m\choose 1} (c_k (1-u)^k t^k)^{m-1} (c_k u^k t^k)^n du&\nonumber
\\ 
&=\frac{t^{k(n+m-1)+\ell}}{\ell!(k!)^{n+m-1}}\int_0^1
\left(a nu^{k(n-1)+\ell}(1-n)^{km}+b m(1-u)^{k(m-1)+\ell}u^{kn}\right)du&\nonumber
\\
&= \frac{t^d}{\ell!(k!)^{(d-\ell)/k}}
\big(a n\Beta(d-km+1,\, km+1)
+b m\Beta(d-kn+1,\ kn+1)\big).&
\label{Cab}
\end{align}
The coefficient $D_{n,m}$ is obtained as the coefficient in $t^d$
when we set $a$ and $b$ to zero.
In other words, $D_{n,m}$ is the coefficient of $t^d$ in the following 
polynomial (in $t$)
\[
\int_0^1
\left( \sum_{i=k}^{\ell-1} c_i u^i t^i  \right)^n
\left(\sum_{i=k}^{\ell-1}c_i (1-u)^i t^i  \right)^m du.
\]
Notice that $Q_{n,m}(t)$ does not involve $a$ or $b$ either, 
because when $a$ or $b$ is involved in the 
multiplication and integration, 
the resulting term 
must be at least of order $t^d.$ So $D_{n,m}$ is the 
coefficient of $t^{d}$ in the above polynomial.
By symmetry, $D_{n,m}=D_{m,n}$, $Q_{n,m}=Q_{m,n}$.
Reversing the roles of $m$ and $n$ we obtain
\[
\frac{f^m*g^n(t)}{t} = P_{m,n}(t) + o(t^d)
= Q_{m,n}(t)+ (C_{m,n}(a, b)+D_{m,n}) t^d + o(t^d),
\]
as $t \to 0$. The assumptions imply that $f^n * g^m  = f^m * g^n$.
We thus have
\[
Q_{n,m}(t)+ (C_{n,m}(a, b)+D_{n,m}) t^d + o(t^d)
= Q_{m,n}(t)+ (C_{m,n}(a, b)+D_{m,n}) t^d + o(t^d),
\]
in a neighbourhood of $0$.
Since $D_{n,m}=D_{m,n}, Q_{n,m}=Q_{m,n}$, 
we obtain
\[
C_{n,m}(a,b)
=
C_{m,n}(a,b).
\]
Looking at the expression for $C_{n,m}$ from 
equation \eqref{Cab} we further obtain
\[
(a-b) \big[n\Beta(k(n-1)+\ell+1, km+1)
-m \Beta(k(m-1)+\ell+1, kn+1)\big]=0.
\]
To conclude that $a=b$ we only have to show that 
the coefficient in the bracket is nonzero.
The ratio of the two terms in the bracket is
\begin{align*}
\frac{n\Beta(k(n-1)+\ell+1, km+1)}{m \Beta(k(m-1)+\ell+1, kn+1)}
&=
\frac{n}{m} \frac{(km)!}{(kn)!} \frac{(kn+\ell-k)!}{(km+\ell-k)!}&
\\
&= \frac{n}{m} \frac{(l-k+1)(l-k+2) \cdots (l-k+kn)}{(l-k+1)(l-k+2) \cdots (l-k+km)}.&
\end{align*}
Since $n \neq m$, assuming that $n>m$ we see that this ratio is strictly bigger 
than $1$. Similarly, when $n<m$ the ratio is strictly smaller than $1$.
\end{proof}
We complement Lemma \ref{lem2} with the comparison of derivatives of all orders at zero.
\begin{corollary}
\label{allorder}
Suppose that $f$ and $g$ are of exponential order.
Let $m, n$ be distinct positive integers. 
Suppose
$H_{n,m}(f,\cdot) = H_{n,m}(g,\cdot)$. Assume $k=I(f)=I(g)<\infty.$ 
If $f^{(k)}(0)=g^{(k)}(0)$, then $f^{(j)}(0)=g^{(j)}(0)$ 
for all $j \ge 0$ for which the
two derivatives exist.
If $f^{(k)}(0)=-g^{(k)}(0)$, then $f^{(j)}(0)=-g^{(j)}(0)$ 
for all $j \ge 0$ for which the
two derivatives exist.
\end{corollary}
\begin{proof}
If $f^{(k)}(0)=g^{(k)}(0)$, by Lemma \ref{lem2}, 
$f^{(j)}(0)=g^{(j)}(0)$ for all $j \ge k$ and 
hence for all $j\geq 0$ for which the
derivatives exist.
If $f^{(k)}(0)=-g^{(k)}(0)$, by Lemma \ref{lemaux}, 
$n-m$ must be even. Then $H_{n,m}(f,\cdot)=H_{n,m}(-g,\cdot)$.
Using $f^{(k)}(0)=(-g)^{(k)}(0)$ and Lemma \ref{lem2}, 
$f^{(j)}(0)=(-g)^{(j)}(0)$ for any $j \ge 0$ for which the
derivatives exist.
\end{proof}

\begin{proof}[\bf\em Proof of Theorem \ref{poly}]
Since $f,g$ are polynomials they are infinitely differentiable and are of
exponential order. Moreover, $I(f)<\infty, I(g)<\infty$. 
By Lemma \ref{lemaux}, $I(f)=I(g)=: k$, say.
Moreover, we have $f^{(k)}(0)=g^{(k)}(0)$, if $n-m$ is odd;
$|f^{(k)}(0)|=|g^{(k)}(0)|$, if $n-m$ is even.  
Suppose first that $n-m$ is odd. 
By Corollary \ref{allorder}, $f^{(j)}(0)=g^{(j)}(0)$ for all $j \ge 0$.
Since polynomials are determined by their derivatives of all 
orders at zero, we have $f$ identical to $g$.
Suppose next that $n-m$ is even. We have two possibilities, i.e., either 
$f^{(k)}(0)=g^{(k)}(0)$ or $f^{(k)}(0)=-g^{(k)}(0)$. Consequently, we have 
either $f^{(j)}(0)=g^{(j)}(0)$ for all $j \ge 0$, 
or $f^{(j)}(0)=-g^{(j)}(0)$ for all $j \ge 0$. 
Hence $f$ is identical to $g$ or identical to $-g$.
\end{proof}

\begin{remark}
\label{entire}
The conclusion of Theorem \ref{poly} remains true if we replace the assumption
that $f, g$ be polynomials by the assumption that they can be analytically extended
to entire functions on the complex plane. In this case, $f,g$ are equal to their
Taylor series which are constructed in the same way for both $f,g$ by the derivatives of all orders at $0$.
By Corollary \ref{allorder}, $f,g$ have the same derivatives of any order
at $0.$ So $f=g.$
So, for example, we know that there is a unique entire function
$f$ such that 
$H_{2,1}(f, \lambda) = 2(\lambda^2+1)/\lambda(\lambda^2+4)$,
and this is, as can be checked, $f(t)= \sin t$.
\end{remark}

We now aim at proving Theorem \ref{Mainthm}.
We need the  preliminary result of Lemma \ref{=>} below.
This lemma is inspired by the approach taken in \cite{LX19}. 

\begin{lemma}\label{=>}
Suppose that $f$ and $g$ are of exponential order, 
\cadlag and nondecreasing
with $f(x)>0, g(x)>0$ for any $x>0$. Let $m, n$ be distinct positive
integers and assume that  $H_{n,m}(f,\cdot)=H_{n,m}(g,\cdot)$.  
Assume further that there exists $a>0$ such that 
$f(x)=g(x)$ for any $x\in[0,a)$ and that $f^{(i)}(a+)$, $g^{(i)}(a+)$ exist
for some $i \ge 0$.
Then $f^{(i)}(a+)=g^{(i)}(a+)$.
\end{lemma}

\begin{proof}
We first prove that $f(a+)=g(a+)$.
Suppose this is not the case. Without loss of generality,
suppose that $f(a+)-g(a+) > 0$.
Then there is positive $\delta$ such that 
\begin{equation}\label{delta}
f(x) - g(x)>0,
\text{ for all } a < x < a+\delta.
\end{equation}
We can assume that $\delta < a$ (else replace $\delta$ by its minimum with $a$).
The assumption that $H_{n,m}(f,\cdot)=H_{n,m}(g,\cdot)$ implies that
\[
Q:= f^n * g^m - f^m * g^n = 0.
\]
That is, $Q$, is identically equal to $0$.
Suppose, without loss of generality, that $n>m$ and write
\begin{align*}
0=Q(a+\delta) &= \left(\int_0^\delta + \int_\delta^{a+\delta}\right)
\left[f(a+\delta-u)^n g(u)^m - g(u)^n f(a+\delta-u)^m\right]\, du
\end{align*}
Let $I_1, I_2$ denote the two integrals.
We have
\[
I_1 = \int_0^\delta f(a+\delta-u)^m  g(u)^m
\left[f(a+\delta-u)^{n-m} - g(u)^{n-m}\right]\, du.
\]
The quantity in the bracket is nonnegative because,
when $0<u<\delta$, we have $a<a+\delta-u<a+\delta$ and, by \eqref{delta},
\[
f(a+\delta-u) > g(a+\delta-u);
\]
furthermore, since we have chosen $\delta < a$ we also have 
$a+\delta-u > a > \delta > u$ and so, by the
monotonicity of $g$, $g(a+\delta-u) \ge g(u)$;
these inequalities show the nonnegativity of the bracketed term.
Hence the integrand in $I_1$
is bounded below by
$g(a+\delta-u)^m  g(u)^m
\left[f(a+\delta-u)^{n-m} - g(u)^{n-m}\right]$.
Hence
\begin{align}
I_1 & > \int_0^\delta g(a+\delta-u)^m  g(u)^m
\left[f(a+\delta-u)^{n-m} - g(u)^{n-m}\right] du\nonumber
\\
&> \int_0^\delta g(a+\delta-u)^m  g(u)^m
\left[g(a+\delta-u)^{n-m} - g(u)^{n-m}\right] du \label{I1}.
\end{align}
On the other hand, since $f=g$ on $[0,a)$,
\begin{align}
I_2 &  = \int_\delta^{a+\delta} f(a+\delta-u)^m  g(u)^m
\left[f(a+\delta-u)^{n-m} - g(u)^{n-m}\right] du \nonumber
\\
&= \int_\delta^{a+\delta} g(a+\delta-u)^m  g(u)^m
\left[g(a+\delta-u)^{n-m} - g(u)^{n-m}\right] du. \label{I2}
\end{align}
It is easy to see that the last integral in \eqref{I1} and the last integral in \eqref{I2}
add up to zero, obtaining that $I_1+I_2 >0$.
So $0=Q(a+\delta)> 0$, which is a  contradiction.

Having proved that $f(a+)=g(a+)$ we now show
that $f^{(i)}(a+)=g^{(i)}(a+)$ for all $i \ge 1$ (provided the
two derivatives exist).
Assume that this is not the case. Let $i_0 \ge 1$ be the least $i$
such that $f^{(i)}(a+) \neq g^{(i)}(a+)$. 
We then have
\[
f^{(i_0)}(a+) \neq g^{(i_0)}(a+), \quad f^{(j)}(a+) = f^{(j)}(a+),
\text{
for all $j=0,\ldots,i_0-1$.}
\]
We can assume, without loss of generality, that
$f^{(i_0)}(a+) > g^{(i_0)}(a+)$.
By Taylor's theorem we then obtain that there is $\delta>0$ 
(which can be taken to be smaller than $a$) such that 
\eqref{delta} holds.
Using exactly the same arguments we arrive at a contradiction.
We thus conclude the proof of the lemma.
\end{proof}

We now pass on to the proof of the main theorem.
\begin{proof}[\bf\em Proof of Theorem \ref{Mainthm}]
If  $f^{(i)}(0)=0$ for all $i\ge 0$ then,
by right analyticity, 
there exists $a>0$ such that $f(x)=0$ for 
all  $x\in[0,a)$. This is in contradiction to the assumption 
that $f(x), g(x)>0$ for all $x>0$. Hence $f^{(j)}(0) \neq 0$ for some $j$.
Similarly, $g^{(j)}(0) \neq 0$ for some $j$. 
 As $f,g$ are nonnegative functions, applying Lemma 
\ref{lemaux} and Corollary \ref{allorder}, we have 
\[
f^{(i)}(0)=g^{(i)}(0), \quad  i\geq 0.
\]
Due to right  analyticity, there exists $a>0$ 
such that $f(x)=g(x)$ for any $x\in[0,a)$. Let 
\[
A:=\sup\{a: f(x)=g(x) \text{ for all } \ x\in[0,a)\}.
\]
Assume that $A<\infty.$ By Lemma \ref{=>} and 
right  analyticity 
\[
f^{(i)}(A)=g^{(i)}(A), \quad i\geq 0.
\]
Again by right analyticity, there exists $h>0$ 
such that $f(x)=g(x)$ for any $x\in[A,A+h)$. This fact is 
in contradiction to the definition of $A$. So we have 
$A=\infty$ which means $f(x)=g(x) $ for all  $x\geq 0.$
\end{proof}


\section{An application: the identification problem in auction theory}
\label{auctions}
There are $N$ bidders for a single item. Bidder $i$ bids $X_i$ units of money.
We assume that $X_1, \ldots, X_N$ are random variables. 
It is important to assume that they are not 
independent because there is a tacit common understanding about
the value of the item. A simple model for this situation, favored by economists 
(see  \cite{LX19}), requires that
\begin{equation*}
X_i = X^* + \epsilon_i, \quad i =1,\ldots, N,
\end{equation*}
where $X^*$ is a random variable representing the common understanding
of the item value. In auction theory, $X^*$ is called ``unobserved heterogeneity''.
The random variable $\epsilon_i$ is the additional value of the item as perceived by
bidder $i$. It is called the ``idiosyncratic part'' of the bid.
Since the bidders act independently, it is reasonable to assume that 
$\epsilon_1, \ldots, \epsilon_N$ are independent random variables.
We also assume that they are independent of $X^*$.
Moreover, we assume that bidders behave identically which means that
the idiosyncratic parts have a common distribution denoted by
\[
F(x) = \P(\epsilon \le x).
\]
The identification problem appearing in auction theory  \cite{LX19} is this: 
Given the distributions of
the two highest bids can we find the distribution of $\epsilon$?
 For more information on the identification problem in auction theory, we refer to, among others, \cite{LV98,LPV00,K11,AH02,GL18,L18, HMS13}.

If $a=(a_1, \ldots, a_N)$ is a finite sequence of real numbers,
let $(a_{(1)}, \ldots, a_{(N)})$ be the sequence obtained by putting
the elements of $a$ in increasing order, that is,
\[
\{a_1, \ldots, a_N\} = \{a_{(1)}, \ldots, a_{(N)}\}, 
\quad a_{(1)} \le  \cdots \le a_{(N)}.
\]
The notation is common in probability theory and mathematical statistics;
see, e.g., \cite[p.\ 321]{RES}.
Thus, $X_{(N)}$ represents the highest and $X_{(N-1)}$ the second highest bid.

The question above then becomes:
if we know the distributions of $X_{(N-1)}$ and $X_{(N)}$ can we find $F$?
Quite clearly, knowledge of the distribution of $X_N$ (which is the same as the
distribution of $X_{N-1}$)  does not imply knowledge of $F$. 
The catch here is that we have information about the highest and second highest bid,
rather than two arbitrary bids; and this is what can possibly lead to  an  affirmative
answer.

To give evidence for the fact that knowledge of the distributions of
$X_{(N)}$ and $X_{(N-1)}$ uniquely specify the distribution of $\epsilon$
we look at a very simple  model:

\paragraph{Example 1}
Suppose $\epsilon$ is an exponential random variable,
that is, $\P(\epsilon>x) = e^{-\theta x}$, $x>0$, with unknown parameter $\theta$.
It is known (and not difficult to see, thanks to the so-called memoryless property
enjoyed  by the exponential random variable) that 
\begin{equation}
\label{eee}
\P(\epsilon_{(N)} > x) = \P(\epsilon_{(N-1)}+\eta > x), \quad x > 0,
\end{equation}
where $\eta$ is an independent copy of $\epsilon$;
see \cite{REN} for a more general version
of this result.  
Since  $X_{(i)} = X^* + \epsilon_{(i)}$, for all $i$, we have
\[
\E e^{-\lambda X_{(N)}}
= \E e^{-\lambda X^*}\, \E e^{-\lambda \epsilon_{(N)}}
= \E e^{-\lambda X^*}\, \E e^{-\lambda \epsilon_{(N-1)}} \E e^{-\lambda \eta}
=  \E e^{-\lambda X_{(N-1)}} \E e^{-\lambda \epsilon}.
\]
We can thus find the Laplace transform of $\epsilon$, and thus the unknown
parameter $\theta$, by dividing the Laplace tranform of 
$X_{(N)}$ with that of $X_{(N-1)}$.
However, this simple argument is possible only for
this simple model, due to the fact that the only continuous random variable with memoryless property 
(and thus the only random variable which gives rise to \eqref{eee}) is the exponential
random variable.

In general, the problem is not as trivial. Theorem \ref{Mainthm} 
answers the identification question affirmatively under some conditions.
We explain this below.

\begin{theorem}
\label{aucthm}
Let $X^*, \epsilon_1, \ldots, \epsilon_N$ be independent positive random variables
where the $\epsilon_i$, $i=1,\ldots,N$ have a common (unknown) 
cumulative distribution function $F$ such that $F$ is right analytic and
$F(x) > 0$ for all $x>0$.
Let $X_{(N)}$, respectively $X_{(N-1)}$, 
be the two largest, respectively second largest, of 
the random variables $X_i=X^*+\epsilon_i$, $i=1,\ldots, N$.
Set
\[
K(F,\lambda) := \frac{\E e^{-\lambda X_{(N)}}}{\E e^{-\lambda X_{(N-1)}}},
\quad \lambda > 0.
\]
Then the function $K(F, \cdot)$ uniquely determines $F$.
\end{theorem}

\begin{proof}
Ordering the $X_i$ is equivalent to ordering the $\epsilon_i$:
\[
X_{(i)} = X^* + \epsilon_{(i)}.
\]
From this we obtain
\begin{equation}
\label{Xe}
K(F,\lambda) = \frac{\E e^{-\lambda X_{(N)}}}{\E e^{-\lambda X_{(N-1)}}}
= \frac{\E e^{-\lambda \epsilon_{(N)}}}{\E e^{-\epsilon_{(N-1)}}}.
\end{equation}
Integrating by parts in a Lebesgue-Stieltjes integral we obtain
\begin{align}
\label{enum}
\E e^{-\lambda \epsilon_{(N)}}
&= \int_{[0,\infty)} e^{-\lambda x} \P(\epsilon_{(N)} \in dx)
= \int_0^\infty \lambda e^{-\lambda x} \P(\epsilon_{(N)} \le x) dx\nonumber
\\&= \int_0^\infty \lambda e^{-\lambda x} F(x)^N dx
= \lambda \widehat{F^N}(\lambda),&
\end{align}
where $\widehat{F^N}$ is the Laplace transform 
of the function $x \mapsto F(x)^N$.
Since
\begin{align*}
\P(\epsilon_{(N-1)} \le x) &= \P(\epsilon_{(N)} \le x) 
- \P(\epsilon_{(N-1)} < x < \epsilon_{(N)})\\
&= F(x)^N - N F(x)^{N-1}(1-F(x))\\
&=  N F(x)^{N-1}- (N-1)F(x)^N,
\end{align*}
we similarly have
\begin{align}
\label{eden}
\E e^{-\lambda \epsilon_{(N-1)}}
&= 
\int_0^\infty \lambda e^{-\lambda x} (N F(x)^{N-1}- (N-1)F(x)^N) dx&\nonumber
\\&= \lambda N \widehat{F^{N-1}}(\lambda) - 
\lambda (N-1) \widehat{F^N}(\lambda).&
\end{align}
Combining \eqref{Xe}, \eqref{enum} and \eqref{eden}, 
we obtain
\[
K(F, \lambda) = \frac{H_{N-1,N}(F,\lambda)}{N+(N-1) H_{N-1,N}(F,\lambda)}.
\]
By Theorem \ref{Mainthm}, $H_{N-1,N}(F,\cdot)$ uniquely determines $F$
and so the same is true for $K(F, \cdot)$.
\end{proof}

\paragraph{Example 2}
Suppose $\epsilon$ has a lognormal distribution with parameters the 
real numbers $\mu$ and $\sigma$, that is, 
$\epsilon = \exp(\sigma \xi - \mu)$, where $\xi$ is 
a standard normal random variable.
Let $K(\mu, \sigma, \lambda) =  
\frac{\E e^{-\lambda \epsilon_{(N)}}}{\E e^{-\epsilon_{(N-1)}}}$.
By Theorem \ref{aucthm}, the mapping $(\mu, \sigma) \mapsto 
K(\mu, \sigma, \cdot)$ is injective. Unlike Example 1, we have no explicit
formula  for the inverse of this mapping.
The reason is the difficulty in obtaining the Laplace transform
of the lognormal distribution (and hence of its powers); see \cite{AJR}.

\section{Open problems}
The question considered in this paper was the injectivity of the 
map $f \mapsto H_{n,m}(f, \cdot)$, where $H_{n,m}(f, \cdot)$
is the ratio of the Laplace transforms of $f^n$ and $f^m$.
Although we answered the question affirmatively when $f$ is a polynomial
(Theorem \ref{poly})
or, more generally, an entire function (see Remark \ref{entire}),
and when $f$ is a nonnegative nondecreasing right analytic function (Theorem \ref{Mainthm}),
we conjecture that the result is much more general.
We can then pose the problem as: show that 
$f \mapsto H_{n,m}(f, \cdot)$ is injective when $f$ ranges over a significantly larger class,
e.g.\ the class of \cadlag functions.

Regarding the application of Section \ref{auctions}, we conjecture that Theorem
\ref{aucthm}, namely, the injectivity of the function $F \mapsto K(F, \cdot)$ remains
true when $F$ ranges over the class of all cumulative distribution functions.

Finally, we ignored completely the inversion problem. An open question then is:
is there an analog of the Laplace inversion formula that determines the function
$f$ from $H_{n,m}(f,\cdot)$ when $n-m$ is odd?





\subsection*{Acknowledgments}
We would like to thank Yao Luo for pointing out this problem. 
We also thank Konstantinos Dareiotis, 
Yao Luo and Yanqi Qiu for reading 
the draft version and providing useful comments.

\vspace*{1cm}

\noindent
\begin{minipage}{\textwidth}
\begin{minipage}{7cm}
Takis Konstantopoulos\\
Department of Mathematical Sciences\\
The University of Liverpool\\ 
Liverpool  L69 7ZL, UK\\
\href{mailto:takiskonst@gmail.com}{takiskonst@gmail.com}
\end{minipage}
\begin{minipage}{7cm}
Linglong Yuan\\
Department of Mathematical Sciences\\
The University of Liverpool\\ 
Liverpool  L69 7ZL, UK\\
\href{mailto:yuanlinglongcn@gmail.com}{yuanlinglongcn@gmail.com}
\end{minipage}
\end{minipage}

\end{document}